\documentclass{gen-j-l}

\newtheorem{theorem}{Theorem}[section]
\newtheorem{lemma}{Lemma}[section]

\theoremstyle{definition}
\newtheorem{definition}{Definition}[section]

\newtheorem{remark}{Remark}[section]

\begin{document}

\title{Admissibility, Stable Units and Connected Components}

\author{Jo\~{a}o J. Xarez}
\address{Departamento de Matem\'{a}tica, Universidade de Aveiro. Campus
Universit\'{a}rio de Santiago. 3810-193 Aveiro. Portugal.}
\email{xarez@ua.pt}
\thanks{The author would like to acknowledge the financial support of {\em Unidade de Investiga\c{c}\~{a}o Matem\'{a}tica e Aplica\c{c}\~{o}es}
 of Universidade de Aveiro, through {\em Programa Operacional Ci\^{e}ncia e Inova\c{c}\~{a}o 2010}
  (POCI 2010) of the {\em Funda\c{c}\~{a}o para a Ci\^{e}ncia e a Tecnologia} (FCT), cofinanced by the European Community fund FEDER}





\begin{abstract}
Consider a reflection from a finitely-complete category $\mathbb{C}$
into its full subcategory $\mathbb{M}$, with unit $\eta
:1_\mathbb{C}\rightarrow HI$. Suppose there is a left-exact functor
$U$ into the category of sets, such that $UH$ reflects isomorphisms
and $U(\eta_C)$ is a surjection, for every $C\in\mathbb{C}$. If, in
addition, all the maps $\mathbb{M}(T,M)\rightarrow
\mathbf{Set}(1,U(M))$ induced by the functor $UH$ are surjections,
where $T$ and $1$ are respectively terminal objects in $\mathbb{C}$
and $\mathbf{Set}$, for every object $M$ in the full subcategory
$\mathbb{M}$, then it is true that: the reflection $H\vdash I$ is
semi-left-exact (admissible in the sense of categorical Galois
theory) if and only if its connected components are ``connected"; it
has stable units if and only if any finite product of connected
components is ``connected". Where the meaning of ``connected" is the
usual in categorical Galois theory, and the definition of connected
component with respect to the ground structure will be given. Note
that both algebraic and topological instances of Galois structures
are unified in this common setting, with respect to categorical
Galois theory.

\end{abstract}

\maketitle

\section{Introduction}
\label{sec:Introduction}

A reflection $H\vdash I$ from a category $\mathbb{C}$ into its full
subcategory $\mathbb{M}$ can be seen as a Galois structure, one in
which all morphisms are taken into account. Hence, such a reflection
is semi-left-exact (in the sense of \cite{CHK:fact}) if and only if
it is an admissible Galois structure (in the sense of categorical
Galois theory). The fundamental theorem of categorical Galois theory
states that, for an admissible Galois structure as above, that is, a
semi-left-exact reflection into a full subcategory, there is an
equivalence $ Spl(E,p)\simeq\mathbb{M}^{Gal(E,p)}$, for every
effective descent morphism $p:E\rightarrow B$ in $\mathbb{C}$,
between the full subcategory $Spl(E,p)$ of the comma category
$(\mathbb{C}\downarrow B)$, determined by the morphisms split by
$p:E\rightarrow B$, and the category $\mathbb{M}^{Gal(E,p)}$ of
actions of the Galois pregroupoid $Gal(E,p)$ in $\mathbb{M}$ (see
\cite{CJKP:stab}). To establish the existence of such equivalences,
that is, in order to prove that the reflection is semi-left-exact,
it is necessary to show, for every $B\in\mathbb{C}$ and every
$(M,g)\in (\mathbb{M}\downarrow I(B))$, that the counit morphism
$\varepsilon_{(M,g)}^B:I^BH^B(M,g)\rightarrow (M,g)$ is an
isomorphism, where $H^B\vdash I^B:(\mathbb{C}\downarrow
B)\rightarrow (\mathbb{M}\downarrow I(B))$ is the induced
adjunction. In the current paper, we prove it is enough to show that
every $\varepsilon_{(T,g)}^B$ is an isomorphism when $T$ is a
terminal object, in order to guarantee semi-left-exactness, provided
there is a (``forgetful'') functor $U$ from $\mathbb{C}$ into sets,
satisfying certain conditions. Such is the case of the two
reflections $\mathbf{CompHaus}\rightarrow \mathbf{Stone}$, compact
Hausdorff spaces into Stone spaces, and $\mathbf{SGr}\rightarrow
\mathbf{SLat}$, semigroups into semilattices, where ``connected
components are connected'' (meaning that the counit morphisms
$\varepsilon_{(T,g)}^B$ are all isomorphisms, which amounts to the
preservation by the reflector of the ``connected component" pullback
diagrams). Furthermore, these two examples are known to satisfy a
stronger condition than semi-left-exactness. In fact, both
reflections $\mathbf{CompHaus}\rightarrow \mathbf{Stone}$ and
$\mathbf{SGr}\rightarrow \mathbf{SLat}$ have stable units (see
\cite{CJKP:stab} and \cite{JLM}, respectively). We will also state
that such a Galois structure with such a ``forgetful'' functor does
have stable units if and only if ``finite products of connected
components are connected''. A connected component is simply the
pullback $C\times_{(\eta_C,\mu)}T$ of a morphism $\mu :T\rightarrow
HI(C)$ from a terminal object $T$ along a unit morphism
$\eta_C:C\rightarrow HI(C)$. Therefore, in our setting,
semi-left-exactness and the stable units property are simplified and
the Galois structures can be classified according to the reflection
of connected components and its products, respectively.

Besides semi-left-exactness and the stable units property, there is
a weaker property and also a stronger one. When the former holds, a
reflection is called simple. A reflection where the latter holds is
called a localization, meaning that the reflector is left-exact,
that is, it preserves finite limits. In our setting, a sufficient
condition, for a reflection to be a localization, will be given on
the connected components. Also, semi-left-exact and simple
reflections are shown to coincide, provided a further condition
holds for the left adjoint $I$.

Finally, the author would like to mention that the results in this
paper had their origin in generalizing the proof of Theorem 3 in
\cite{JLM}, where it is shown that the reflection of semigroups into
semilattices has stable units.\footnote{The property that
``connected components are connected", i.e., semi-left-exactness in
our setting, was called attainability in \cite{Tamura}, in the
particular case of semigroups.}

\section{Ground Structure}
\label{sec:Ground Structure}

In this section \ref{sec:Ground Structure}, it is given the setting
in which all the propositions of the current paper hold.

Consider an adjunction $H\vdash I:\mathbb{C}\rightarrow \mathbb{M}$,
with unit $\eta :1_\mathbb{C}\rightarrow HI$, such that the category
$\mathbb{C}$ has finite limits and the right adjoint $H$ is a full
inclusion of $\mathbb{M}$ in $\mathbb{C}$, i.e., the adjunction is a
reflection of the category $\mathbb{C}$ into its full subcategory
$\mathbb{M}$. Consider as well a functor $U:\mathbb{C}\rightarrow
\mathbf{Set}$ from $\mathbb{C}$ into the category of sets, with the
following properties:

\begin{enumerate}

\item [$(a)$] $U$ is left exact (i.e., $U$ preserves finite limits);

\item [$(b)$] $UH$ reflects isomorphisms;

\item [$(c)$] every map $U(\eta_C):U(C)\rightarrow UHI(C)$ is a
surjection, for every unit morphism $\eta_C$ of the reflection
above, $C\in\mathbb{C}$;

\item [$(d)$] every map $\mathbb{C}(T,M)\rightarrow\mathbf{Set}(U(T),U(M))$, which is the restriction of the functor $U$
to the hom-set $\mathbb{C}(T,M)$, is a surjection, for any object
$M\in\mathbb{M}$, with $T$ a terminal object in $\mathbb{C}$.

\end{enumerate}

\begin{remark}\label{remark:HI=1}

It is convenient, without no loss of generality, to chosen the unit
$\eta :1_\mathbb{C}\rightarrow HI$ so that the counit is an identity
$IH=1_\mathbb{M}$.

\end{remark}

\begin{remark}\label{remark:T in M}

It is also convenient to assume, without no loss of generality, that
$T$ is a terminal object chosen to be in $\mathbb{M}$. In such case,
$\mathbb{C}(T,M)=\mathbb{M}(T,M)$ in $(d)$.\footnote{Recall that a
full reflective subcategory $\mathbb{M}$ of $\mathbb{C}$ is closed
for limits in $\mathbb{C}$.}

\end{remark}

\begin{remark}\label{remark:split monic delta_T=>(d)}

Suppose $UH$ has a left adjoint $F$, being the counit morphism of
such an adjunction $\delta :F(UH)\rightarrow 1_\mathbb{M}$. If the
counit morphism of a terminal object $\delta_T:F(UH)(T)\rightarrow
T$ is a split monomorphism then condition $(d)$ necessarily holds.
Notice that all functors $UH$, considered in any instance of the
ground structure presented in last section \ref{sec:Examples}, have
a left adjoint, and the respective counit morphisms $\delta_T$ of
terminal objects are isomorphisms, i.e., $F$ preserves the terminal
objects in $\mathbf{Set}$.\footnote{Notice that any counit morphism
is an isomorphism if it is a monomorphism, provided the right
adjoint reflects isomorphisms.}

\end{remark}

\section{Properties of the Reflection}
\label{sec:Properties of the Reflection}

It is to be defined when the reflection $I\dashv H$ is 1.
\textit{simple}, 2. \textit{semi-left-exact} or 3. \textit{to have
stable units} (notions introduced in \cite{CHK:fact}). One easily
checks from the definitions below that if $I$ is a left-exact
functor, in which case the reflection is called a
\textit{localization}, then 1., 2. and 3. hold, and that 3. is
stronger than 2., which in turn is stronger than 1. ($I$ is left
exact $\Rightarrow I\dashv H$ has stable units $\Rightarrow I\dashv
H$ is semi-left-exact $\Rightarrow I\dashv H$ is simple). The
semi-left-exactness is also called \textit{admissibility} in
categorical Galois theory (see \cite{CJKP:stab}).

\begin{definition}\label{def:simple reflection}
The reflection $I\dashv H$ is called simple if the morphism
$I(w):I(A)\rightarrow I(C)$ is an isomorphism in every diagram of
the form
\\
\\
\\
\\
\\

\begin{equation}\label{eq:simple reflection}
\vcenter{ $$
\begin{picture}(50,40)(0,0)
\put(0,0){$B$}\put(0,50){$C$}\put(100,0){$HI(B)$\hspace{10pt,}}\put(100,50){$HI(A)$}
\put(120,25){$HI(f)$} \put(50,10){$\eta_B$}
\put(15,3){\vector(1,0){80}}\put(15,53){\vector(1,0){80}}
\put(3,45){\vector(0,-1){35}}\put(115,45){\vector(0,-1){35}}
\put(-50,100){$A$}\put(-35,50){$f$}\put(25,90){$\eta_A$}
\put(-10,72){$w$}
\put(-40,98){\vector(1,-1){40}}\put(-45,95){\vector(1,-2){43}}\put(-40,103){\vector(3,-1){135}}
\end{picture}
$$\\
}
\end{equation}

\noindent where the rectangular part of the diagram is a pullback
square, $\eta_A$ and $\eta_B$ are unit morphisms, and $w$ is the
unique morphism which makes the diagram commute.

\end{definition}

\begin{remark}\label{remark:simple reflection}

The functor between comma categories $I^B:(\mathbb{C}\downarrow
B)\rightarrow (\mathbb{M}\downarrow I(B))$, sending $f:A\rightarrow
B$ to $I(f)$, has a right adjoint $H^B$ sending $g:M\rightarrow
I(B)$ to its pullback along $\eta_B:B\rightarrow HI(B)$, for each
$B\in\mathbb{C}$. Hence, $I\dashv H$ is simple if and only if
$I^B\eta^B$ is an isomorphism for every $B\in\mathbb{C}$, where
$\eta^B$ is the unit of the adjunction $I^B\dashv H^B$
(equivalently, $\varepsilon^BI^B$ is an isomorphism for every
$B\in\mathbb{C}$, where $\varepsilon^B$ is the counit of $I^B\dashv
H^B$).

\end{remark}

\begin{definition}\label{def:semi-left-exact reflection}

The reflection $I\dashv H$ is called semi-left-exact, or admissible,
if the left adjoint $I$ preserves all pullback squares of the form
\\
\\

\begin{equation}\label{eq:semi-left-exact reflection}
\vcenter{ $$
\begin{picture}(100,20)
\put(0,0){$C$}\put(-25,50){$C\times_{HI(C)}M$}\put(80,0){$HI(C)$\hspace{10pt,}}\put(90,50){$M$}
\put(-8,25){$\pi_1$}\put(100,25){$g$}
\put(40,10){$\eta_C$}\put(50,60){$\pi_2$}
\put(15,3){\vector(1,0){60}}\put(33,53){\vector(1,0){50}}
\put(4,45){\vector(0,-1){33}}\put(95,45){\vector(0,-1){33}}
\end{picture}
$$\\
}
\end{equation}

\noindent where the bottom arrow $\eta_C$ is a unit morphism, and
the object $M$, in the upper corner to the right, is in the
subcategory $\mathbb{M}$.

\end{definition}

\begin{remark}\label{remark:semi-left-exact reflection}
The reflection $I\dashv H$ is semi-left-exact if and only if the
functor $I$ preserves all pullback squares in which the arrow in the
right edge is in the subcategory $\mathbb{M}$, as it is easy to
prove. Equivalently, $I\dashv H$ is semi-left-exact if and only if
the right adjoint $H^B$ is fully faithful ($\varepsilon^B$ is an
isomorphism) for every $B\in\mathbb{C}$. Therefore, the reflection
is simple if it is semi-left-exact (cf. remark \ref{remark:simple
reflection}).

\end{remark}

\begin{definition}\label{def:stable units}

The reflection $I\dashv H$ has stable units if the left adjoint $I$
preserves all pullback squares of the form
\\
\\
\\

\begin{equation}\label{eq:stable units}
\vcenter{ $$
\begin{picture}(100,20)
\put(0,0){$C$}\put(-25,50){$C\times_{HI(C)}D$}\put(80,0){$HI(C)$\hspace{10pt,}}\put(90,50){$D$}
\put(-10,25){$\pi_1$}\put(100,25){$g$}
\put(45,10){$\eta_C$}\put(45,60){$\pi_2$}
\put(15,3){\vector(1,0){63}}\put(30,53){\vector(1,0){48}}
\put(3,45){\vector(0,-1){33}}\put(95,45){\vector(0,-1){33}}
\end{picture}
$$\\
}
\end{equation}

\noindent in which the bottom arrow $\eta_C$ is a unit morphism.

\end{definition}

\begin{remark}\label{remark:stable units}

One could also show that the reflection $I\dashv H$ has stable units
if and only if the left adjoint $I$ preserves all pullback squares
in which the object at the right corner in the bottom belongs to the
subcategory $\mathbb{M}$.

\end{remark}

\section{Admissibility and Connected Components}
\label{sec:Admissibility and Connected Components}

\begin{definition}\label{def:connected component}

Consider any morphism $\mu :T\rightarrow HI(C)$ from a terminal
object $T$ into $HI(C)$, for some $C\in\mathbb{C}$.

The connected component of the morphism $\mu$, with respect to the
ground structure of section \ref{sec:Ground Structure}, is the
pullback $C_\mu =C\times_{HI(C)}T$ in the following pullback square
\\
\\
\\

\begin{equation}\label{eq:connected component}
\vcenter{ $$
\begin{picture}(100,20)
\put(0,0){$C$}\put(0,50){$C_\mu$}\put(80,0){$HI(C)$\hspace{10pt.}}\put(90,50){$T$}
\put(-15,25){$\pi_1^\mu$}\put(100,25){$\mu$}
\put(40,10){$\eta_C$}\put(40,60){$\pi_2^\mu$}
\put(15,3){\vector(1,0){60}}\put(20,53){\vector(1,0){54}}
\put(3,45){\vector(0,-1){33}}\put(95,45){\vector(0,-1){33}}
\end{picture}
$$\\
}
\end{equation}

\end{definition}

The following Theorem \ref{theorem:admissibility=connected
components connected} states that, under the assumptions given in
section \ref{sec:Ground Structure}, in order to prove the
semi-left-exactness of the full reflection $I\dashv H$, one has only
to establish the preservation by $I$ of the pullback squares like
those in diagram (\ref{eq:semi-left-exact reflection}) in which the
object $M$ is terminal. So, in our context, semi-left-exactness
reduces to connected components being ``connected'', in the sense
$HI(C_\mu)\cong T$. Notice that $HI(C_\mu)\cong T$ if and only if
$I(\pi_2^\mu)$ is an isomorphism in diagram (\ref{eq:connected
component}), since $HI(T)\cong T$.\\

The following Lemma \ref{lemma:surjection => injectivity in Set},
which states a trivial result in sets, will be needed in the proofs
of the ``if parts'' of Theorems \ref{theorem:admissibility=connected
components connected} and \ref{theorem:stable units=product of
connected components connected}.

\begin{lemma}\label{lemma:surjection => injectivity in Set}

Let $gf$ be the composite of a pair $f:A\rightarrow B$,
$g:B\rightarrow C$ of surjections in the category of sets. Consider
the pullback $pr_1:f^{-1}g^{-1}(\{c\})\rightarrow A$ of the function
$\hat{c}:\{\ast\}\rightarrow C$, $\hat{c}(\ast)=c$, along the
function $gf:A\rightarrow C$, for any element $c\in C$ (see diagram
(\ref{eq:surjection => injectivity in Set}) below). Then, the
function $g$ is an injection if and only if, for every element $c\in
C$, $fw=\hat{b}!$ for some function $\hat{b}:\{\ast\}\rightarrow B$
(i.e., $fw$ factorises through a one point set), where $!$ denotes
the unique function into $\{\ast\}$.
\\
\\
\\

\begin{equation}\label{eq:surjection => injectivity in Set}
\vcenter{ $$
\begin{picture}(100,20)
\put(0,0){$A$}\put(-20,50){$f^{-1}g^{-1}(\{c\})$}\put(85,0){B}\put(170,0){$C$}\put(167,50){$\{\ast\}$}
\put(-12,25){$pr_1$}\put(180,25){$\hat{c}$}\put(40,-8){$f$}\put(130,-8){$g$}\put(83,57){$pr_2$}
\put(15,3){\vector(1,0){60}}\put(100,3){\vector(1,0){60}}\put(40,53){\vector(1,0){120}}
\put(3,45){\vector(0,-1){33}}\put(175,45){\vector(0,-1){33}}

\put(83,30){$\{\ast\}$}\put(90,25){\vector(0,-1){13}}\put(95,15){$\hat{b}$}\put(35,45){\vector(4,-1){40}}
\end{picture}
$$\\
}
\end{equation}

\end{lemma}

\begin{theorem}\label{theorem:admissibility=connected components connected}
Under the assumptions of section \ref{sec:Ground Structure}, the
full reflection $I\dashv H$ is semi-left-exact if and only if
$HI(C_\mu)\cong T$, for every connected component $C_\mu$, where $T$
is any terminal object.

\end{theorem}
\begin{proof}

If $I\dashv H$ is semi-left-exact then, by Definition
\ref{def:semi-left-exact reflection}, $I(C\times_{HI(C)}M)$ must be
isomorphic to $I(M)$ in diagram (\ref{eq:semi-left-exact
reflection}), since $I(\eta_C)$ is an isomorphism.\footnote{$\
\varepsilon_{I(C)}I(\eta_C)=1_{I(C)}$, where $\varepsilon
:IH\rightarrow 1_\mathbb{M}$ is the counit of the full reflection
and therefore an isomorphism.} In particular,
$I(C\times_{HI(C)}M)\cong I(T)$ if $M\cong T$.

Suppose now that every connected component is connected, that is,
$I(C_\mu)\cong T$ for every $\mu :T\rightarrow HI(C)$,
$C\in\mathbb{C}$, and consider the diagram:
\\
\\
\\
\\

\begin{equation}\label{eq:admissibility=connected components connected}
\vcenter{ $$
\begin{picture}(250,45)
\put(155,107){$pr_2$}\put(-20,50){$C\times_{HI(C)}M$}\put(-5,100){$C_{g\mu}$}\put(128,50){$HI(C\times_{HI(C)}M)$}\put(295,50){$M$}\put(298,100){$T$}
\put(-12,75){$pr_1$}\put(305,75){$\mu$}\put(60,42){$\eta_{C\times_{HI(C)}M}$}\put(230,42){$HI(\pi_2)$}
\put(40,53){\vector(1,0){80}}\put(210,53){\vector(1,0){80}}\put(20,103){\vector(1,0){270}}
\put(3,95){\vector(0,-1){33}}\put(300,95){\vector(0,-1){33}}

\put(73,80){$HI(C_{g\mu})$}\put(120,80){\vector(2,-1){35}}\put(135,75){$HI(pr_1)$}\put(15,95){\vector(4,-1){50}}\put(40,93){$\eta_{C_{g\mu}}$}

\put(0,0){$C$}\put(155,0){$HI(C)$}\put(285,0){$HI(C)$\hspace{10pt}.}
\put(-10,25){$\pi_1$}\put(305,25){$g$}\put(70,-8){$\eta_C$}\put(230,-8){$1_{HI(C)}$}
\put(15,3){\vector(1,0){130}}\put(190,3){\vector(1,0){90}}
\put(3,45){\vector(0,-1){33}}\put(300,45){\vector(0,-1){33}}

\put(170,25){$HI(\pi_1)$}\put(165,45){\vector(0,-1){33}}

\end{picture}
$$\\
}
\end{equation}

The bottom rectangle in diagram (\ref{eq:admissibility=connected
components connected}) is a pullback square of the form
(\ref{eq:semi-left-exact reflection}), since
$HI(\pi_2)\eta_{C\times_{HI(C)}M}=\eta_M\pi_2$ and $\eta_M$ is an
identity, because $M\in\mathbb{M}$ (cf. remark \ref{remark:HI=1}).
According to $(a)$, $(b)$ and $(c)$ in section \ref{sec:Ground
Structure}, the reflection $I\dashv H$ is semi-left-exact if and
only if $UHI(\pi_2)$ is an injection in $\mathbf{Set}$, in every
diagram (\ref{eq:admissibility=connected components connected}). The
upper rectangle in diagram (\ref{eq:admissibility=connected
components connected}) (associated to the equation $\mu
pr_2=HI(\pi_2)\eta_{C\times_{HI(C)}M}pr_1$) is a pullback square,
therefore the outer rectangle in diagram
(\ref{eq:admissibility=connected components connected}) is in fact a
pullback square of the form (\ref{eq:connected component}), and
$C_{g\mu}$ is the connected component associated to $g\mu
:T\rightarrow HI(C)$. Then, as $(d)$ in section \ref{sec:Ground
Structure} holds, by Lemma \ref{lemma:surjection => injectivity in
Set}, $UHI(\pi_2)$ is an injection since every connected component
is connected, in particular $HI(C_{g\mu})\cong T$, for any morphisms
$g:M\rightarrow HI(C)$, with $M\in\mathbb{M}$, and $\mu
:T\rightarrow M$, with $T$ terminal.

\end{proof}

\section{Stable Units Property and Product of Connected Components}
\label{sec:Product of Connected Components and the Stable Units
Property}

\begin{theorem}\label{theorem:stable units=product of connected components connected}
Under the assumptions of section \ref{sec:Ground Structure}, the
full reflection $I\dashv H$ has stable units if and only if
$HI(C_\mu\times D_{\nu})\cong T$, for every pair of connected
components $C_\mu$, $D_\nu$, where $T$ is any terminal object.

\end{theorem}

\begin{proof}

If $I\dashv H$ has stable units then the functor $I$ preserves
finite products, since a product diagram is a pullback square in
which the right corner in the bottom is a terminal object
$T\in\mathbb{M}$ (cf. remark \ref{remark:stable units}). Therefore,
$HI(C_\mu\times D_{\nu})\cong T$ since $HI(C_\mu)\cong T\cong
HI(D_\nu)$, by Theorem \ref{theorem:admissibility=connected
components connected}, for every pair of connected components
$C_\mu$, $D_\nu$.

Suppose now that every product of two connected components is
connected, i.e., $HI(C_\mu\times D_{\nu})\cong T$ for every pair of
morphisms $\mu :T\rightarrow HI(C)$ and $\nu: T\rightarrow HI(D)$,
$C,D\in \mathbb{C}$, and consider the diagram:
\\
\\
\\
\\
\\
\\

\begin{equation}\label{eq:stable units=product of connected components connected}
\vcenter{ $$
\begin{picture}(250,80)
\put(155,108){$\pi_2$}\put(45,100){$C\times_{HI(C)}D$}\put(128,50){$HI(C\times_{HI(C)}D)$}
\put(295,50){$HI(D)$}\put(298,100){$$}
\put(305,75){$\eta_D$}\put(230,57){$HI(\pi_2)$}\put(295,100){$D$}
\put(205,53){\vector(1,0){80}}\put(105,103){\vector(1,0){185}}
\put(68,95){\vector(0,-1){83}}\put(300,95){\vector(0,-1){33}}

\put(65,0){$C$}\put(155,0){$HI(C)$}\put(285,0){$HI(C)$\hspace{10pt}.}
\put(55,50){$\pi_1$}\put(305,25){$HI(g)$}\put(110,-8){$\eta_C$}\put(230,-8){$1_{HI(C)}$}
\put(80,3){\vector(1,0){65}}\put(190,3){\vector(1,0){90}}
\put(300,45){\vector(0,-1){33}}

\put(170,25){$HI(\pi_1)$}\put(165,45){\vector(0,-1){33}}\put(100,90){\vector(2,-1){50}}\put(128,80){$\eta_{C\times_{HI(C)}D}$}

\put(-40,0){$C_{HI(g)\nu}$}\put(-5,3){\vector(1,0){65}}\put(15,-13){$\pi_1^{HI(g)\nu}$}
\put(-60,150){$C_{HI(g)\nu}\times
D_\nu$}\put(-30,140){\vector(0,-1){125}}\put(-25,80){$p_1$}
\put(295,150){$D_\nu$}\put(300,145){\vector(0,-1){30}}\put(305,130){$\pi_1^\nu$}\put(0,155){\vector(1,0){290}}
\put(140,160){$p_2$}\put(0,145){\vector(2,-1){65}}\put(40,130){$w$}

\end{picture}
$$\\}
\end{equation}

The inside rectangle in diagram (\ref{eq:stable units=product of
connected components connected}) is a pullback square of the form
(\ref{eq:stable units}), since $HI(g)\eta_D =\eta_{HI(C)}g$ and
$\eta_{HI(C)}$ is an identity, because $HI(C)\in\mathbb{M}$ (cf.
remark \ref{remark:HI=1}).

According to $(a)$, $(b)$ and $(c)$ in section \ref{sec:Ground
Structure}, the reflection $I\dashv H$ has stable units if and only
if $UHI(\pi_2)$ is an injection in $\mathbf{Set}$, for every diagram
of the form (\ref{eq:stable units}). In fact, $UHI(\pi_2)$ is
obviously a surjection, since
$UHI(\pi_2)U(\eta_{C\times_{HI(C)}D})=U(\eta_D)U(\pi_2)$ and
$U(\eta_{C\times_{HI(C)}D})$, $U(\eta_D)$ and $U(\pi_2)$ are all
surjections by the assumptions in section \ref{sec:Ground
Structure}.
 The morphisms $p_1$ and $p_2$ in diagram (\ref{eq:stable units=product of connected components
 connected}) are the product projections of the product of the
 connected components $C_{HI(g)\nu}$ and $D_\nu$. The morphism $w$
 is the unique morphism which makes diagram (\ref{eq:stable units=product of connected components
 connected}) commute; it is well defined since

\begin{center}

 $HI(g)\eta_D\pi_1^\nu p_2=HI(g)\nu\pi_2^\nu p_2=$\\
 $=HI(g)\nu\pi_2^{HI(g)\nu}p_1$\\
 (because both $\pi_2^\nu p_2$ and
 $\pi_2^{HI(g)\nu}p_1$ have the same domain and codomain,\\ the
 latter being the terminal object $T$)\\
$=\eta_C\pi_1^{HI(g)\nu}p_1$.\\

\end{center}

Then, as $(d)$ in section \ref{sec:Ground Structure} holds, by Lemma
\ref{lemma:surjection => injectivity in Set}, $UHI(\pi_2)$ is an
injection if the outer rectangle in the following diagram is a
pullback square, for every morphism $\nu :T\rightarrow HI(D)$ from
the terminal object into $HI(D)$ (cf. diagram (\ref{eq:surjection =>
injectivity in
Set})):
\\
\\
\\

\begin{equation}\label{eq:stable units:surjection => injectivity in Set}
\vcenter{ $$
\begin{picture}(200,20)
\put(-20,0){$C\times_{HI(C)}D$}\put(-20,50){$C_{HI(g)\nu}\times
D_\nu$}\put(40,3){\vector(1,0){35}}\put(165,3){\vector(1,0){35}}\put(170,-8){$HI(\pi_2)$}\put(205,0){$HI(D)$\hspace{10pt}.}
\put(215,45){\vector(0,-1){33}}\put(212,50){$T$}\put(220,25){$\nu$}\put(70,58){$p_2$}
\put(45,53){\vector(1,0){65}}\put(120,50){$D_\nu$}\put(160,58){$\pi_2^\nu$}\put(140,53){\vector(1,0){65}}
\put(120,25){\vector(0,-1){13}}\put(125,15){$HI(w)$}
\put(20,32){$\eta_{C_{HI(g)\nu}\times D_\nu}$}

\put(85,0){$HI(C\times_{HI(C)}D)$}
\put(-10,25){$w$}\put(40,-8){$\eta_{C\times_{HI(C)}D}$}

\put(3,45){\vector(0,-1){33}}

\put(83,30){$HI(C_{HI(g)\nu}\times
D_\nu)$}\put(35,45){\vector(4,-1){40}}
\end{picture}
$$\\
}
\end{equation}

In order to show that the outer rectangle in diagram (\ref{eq:stable
units:surjection => injectivity in Set}) is a pullback square,
consider a morphism $l:A\rightarrow C\times_{HI(C)}D$ such that
$HI(\pi_2)\eta_{C\times_{HI(C)}D}l=\nu !$. Let $\bar{l}=\langle
l_1,l_2 \rangle:A\rightarrow C_{HI(g)\nu}\times D_\nu$ be the
morphism into the product of the two connected components, in which
$l_1:A\rightarrow C_{HI(g)\nu}$ and $l_2:A\rightarrow D_\nu$ are the
morphisms determined in the pullback squares of the connected
components by $\pi_1^{C_{HI(g)\nu}}l_1=\pi_1l$ and
$\pi_1^{D_\nu}l_2=\pi_2l$, respectively. It is then a routine
calculation to verify that $w$ is a monomorphism and $w\bar{l}=l$.

\end{proof}

\begin{remark}\label{remark:stable units=product of connected components
connected}

It is an immediate consequence of Theorems \ref{theorem:stable
units=product of connected components connected} and
\ref{theorem:admissibility=connected components connected} that,
provided the preservation of finite products by the left adjoint $I$
is added to the assumptions of section \ref{sec:Ground Structure},
the reflection $I\dashv H$ has stable units if and only if it is
semi-left-exact.

\end{remark}

\section{Left-Exactness and Pullbacks of Connected Components}

The following Theorem \ref{theorem:pullbacks of connected components
connected=>left-exactness} gives a sufficient condition for the
reflection $I\dashv H$ to be a localization, that is, for the left
adjoint $I$ to be left exact (see section \ref{sec:Properties of the
Reflection}).

\begin{theorem}\label{theorem:pullbacks of connected components connected=>left-exactness}
Under the assumptions of section \ref{sec:Ground Structure}, the
full reflection $I\dashv H$ is a localization if
$HI(A_\mu\times_CB_\nu)\cong T$, for every pullback
$A_\mu\times_CB_\nu$ of any pair of connected components $A_\mu$,
$B_\nu$, where $T$ is any terminal object. That is, the left adjoint
is left exact if every pullback of connected components is
connected.

\end{theorem}

\begin{proof}
Consider the diagram
\\
\\
\\
\\
\\
\\
\\
\\
\\
\\
\\
\\
\\
\\
\\

\begin{equation}\label{eq:pullbacks of connected components connected=>left-exactness}
\vcenter{ $$
\begin{picture}(250,125)

\put(0,0){$A_\mu$}\put(15,5){\vector(1,0){45}}\put(30,-8){$\pi_1^\mu$}\put(68,0){$A$}
\put(80,10){\vector(1,1){25}}\put(80,30){$\eta_A$}\put(80,5){\vector(1,0){150}}\put(155,-8){$f$}\put(240,0){$C$\hspace{10pt},}
\put(235,10){\vector(-1,1){25}}\put(225,30){$\eta_C$}

\put(243,190){\vector(0,-1){170}}\put(250,105){$g$}\put(240,200){$B$}
\put(235,195){\vector(-1,-1){20}}\put(215,188){$\eta_B$}
\put(243,245){\vector(0,-1){30}}\put(250,230){$\pi_1^\nu$}\put(240,255){$B_\nu$}

\put(-15,255){$A_\mu\times_CB_\nu$}\put(5,245){\vector(0,-1){225}}\put(-15,130){$p_1$}\put(30,260){\vector(1,0){200}}\put(118,268){$p_2$}
\put(30,245){\vector(1,-1){30}}\put(45,235){$j$}

\put(55,200){$A\times_CB$}\put(95,205){\vector(1,0){135}}\put(155,212){$\pi_2$}\put(70,190){\vector(0,-1){170}}\put(55,105){$\pi_1$}
\put(90,195){\vector(1,-1){20}}\put(100,190){$\eta_{A\times_CB}$}

\put(93,40){$HI(A)$}\put(145,30){$HI(f)$}\put(130,45){\vector(1,0){60}}
\put(195,40){$HI(C)$}\put(195,160){$HI(B)$}\put(210,150){\vector(0,-1){90}}\put(213,105){$HI(g)$}
\put(80,160){$HI(A\times_CB)$}\put(110,150){\vector(0,-1){90}}\put(75,105){$HI(\pi_1)$}
\put(145,170){$HI(\pi_2)$}\put(140,165){\vector(1,0){50}}\put(125,150){\vector(1,-1){20}}\put(135,145){$w$}

\put(112,120){$HI(A)\times_{HI(C)}HI(B)$}\put(165,130){\vector(1,1){25}}\put(150,110){\vector(-1,-2){28}}

\end{picture}
$$\\
}
\end{equation}

\noindent wherein $A_\mu\times_CB_\nu =A_\mu\times_{(f\pi_1^\mu
,g\pi_1^\nu )}B_\nu$ and
$HI(A)\times_{HI(C)}HI(B)=HI(A)\times_{(HI(f),HI(g))}HI(B)$ are
pullbacks, and $j$ and $w$ are the unique morphisms making the
diagram commute.

One has to prove that $U(w)$ is always a bijection. It follows from
$I(A_\mu\times_CB_\nu)\cong T$ that
$U(A_\mu\times_CB_\nu)\neq\emptyset$, for all connected components
$A_\mu$, $B_\nu$, which implies that $U(w)$ is a surjection, under
the assumptions of section \ref{sec:Ground Structure}. Note that
$U(\eta_{A\times_CB})^{-1}U(w)^{-1}(U(A_\mu),U(B_\nu))=U(A_\mu\times_CB_\nu)$
in $\mathbf{Set}$, which implies that $U(w)$ is an injection, since
$UHI(j)U(\eta_{A_\mu\times_CB_\nu})=U(\eta_{A\times_CB})U(j)$ and
$UHI(A_\mu\times_CB_\nu)=\{\ast\}$.
\end{proof}

\section{Admissibility of a Simple Reflection}
\label{sec:Admissibility of a Simple Reflection}

\begin{theorem}\label{theorem:Admissibility of a Simple Reflection}
Let the following condition and all assumptions of section
\ref{sec:Ground Structure} hold:

\noindent $(e)$ every map
$I_{T,C}:\mathbb{C}(T,C)\rightarrow\mathbb{M}(T,I(C))$, the
restriction of the reflector $I$ to the hom-set $\mathbb{C}(T,C)$,
is a surjection, for every object $C\in\mathbb{C}$, with $T=HI(T)$ a
terminal object in $\mathbb{C}$. Then, the reflection $I\dashv H$ is
semi-left-exact if and only if it is simple.

\end{theorem}

\begin{proof}

Suppose that $I\dashv H$ is a simple reflection, that is, $I(w)$ is
an isomorphism in every diagram of the form (\ref{eq:simple
reflection}) in Definition \ref{def:simple reflection}, and consider
the pullback square (\ref{eq:connected component}) in Definition
\ref{def:connected component}. Let $w:T\rightarrow C_\mu$ be the
unique morphism such that $\pi_1^\mu w=\nu$ and $\pi_2^\mu w=1_T$,
where $\nu$ is such that $HI(\nu)=\mu$ ($\nu$ exists by $(e)$ in the
statement). Note that the composite $I(\pi_2^\mu )I(w)$ is the
isomorphism $1_T$. Therefore, $I(\pi_2^\mu )$ is an isomorphism,
since $I(w)$ is an isomorphism by assumption.

\end{proof}

\section{Examples}
\label{sec:Examples}

\noindent 1. Consider the full reflection of compact Hausdorff
spaces into Stone spaces $H\vdash I:\mathbf{CompHaus}\rightarrow
\mathbf{Stone}$, where each unit map $\eta_X:X\rightarrow HI(X)$ is
the canonical projection of $X$ into the set of its components, this
set being given the quotient topology with respect to $\eta_X$.
Hence, condition $(c)$ in section \ref{sec:Ground Structure} holds
for the functor $U$ which forgets the topology. Conditions $(a)$ and
$(b)$ of section \ref{sec:Ground Structure} hold as well since
$U:\mathbf{CompHaus}\rightarrow\mathbf{Set}$ is monadic, and
condition $(d)$ holds trivially. This reflection is known to have
stable units, therefore finite products of connected components are
connected.

Let $\hat{0}:T\rightarrow [0,1]$ and $\hat{1}:T\rightarrow [0,1]$ be
the two obvious inclusions of the one point topological space into
the closed interval of real numbers $[0,1]$, with the usual
topology. Then, the pullback $T\times_{(\hat{0},\hat{1})}T
=\emptyset$ is the empty space, not connected in our sense, being
clear that the reflector $I$ is not left exact, since it does not
preserve the pullback diagram of $\hat{0}$ and $\hat{1}$, and also
that the sufficient condition of Theorem \ref{theorem:pullbacks of
connected components connected=>left-exactness} does not hold.
\\

\noindent 2. With the exception of $(d)$, every assumption of
section \ref{sec:Ground Structure} hold for any reflection from a
variety of universal algebras into one of its subvarieties, provided
with the forgetful functor into $\mathbf{Set}$. Notice that, for
these reflections, condition $(d)$ of section \ref{sec:Ground
Structure} is equivalent to idempotency of the algebras in the
subvariety, meaning that every element of an algebra in the
subvariety is a subalgebra.

In particular, it is easy to check that condition $(d)$ in section
\ref{sec:Ground Structure} holds for the reflection $H\vdash I:
\mathbf{SGr}\rightarrow \mathbf{SLat}$ of semigroups into
semilattices, which is known to have stable units (see \cite{JLM}).
Therefore, all finite products of connected components are
connected.

The additive semigroup $\mathbb{N}$ of non-negative integers has two
connected components, $\{0\}$ and $\{1,2,3,...\}$, with respect to
the reflection $\mathbf{SGr}\rightarrow \mathbf{SLat}$. The pullback
of the inclusions $\{0\}\rightarrow \mathbb{Z}$ and
$\{1,2,3,...\}\rightarrow \mathbb{Z}$ into the integers is the empty
semigroup $\emptyset$, which is not connected
($I(\emptyset)=\emptyset$ is not terminal). Hence, this reflection
is not a localization, and also the sufficient condition of Theorem
\ref{theorem:pullbacks of connected components
connected=>left-exactness} does not hold.\\

The reflection $H\vdash I: \mathbf{SGr}\rightarrow \mathbf{Band}$ of
semigroups into bands\footnote{A semigroup is called a band if every
one of its elements is idempotent.} is not a semi-left-exact
reflection (cf. \cite{JLM}). Notwithstanding, all assumptions in
section \ref{sec:Ground Structure} hold for this reflection;
therefore not every connected component is connected, by Theorem
\ref{theorem:admissibility=connected components connected} (see
Example 7 in \cite{JLM}).\\

Note that Theorem \ref{theorem:Admissibility of a Simple Reflection}
holds for the reflection $H\vdash I: \mathbf{Band}\rightarrow
\mathbf{SLat}$ from bands into semilattices (a subreflection of
$\mathbf{SGr}\rightarrow \mathbf{SLat}$).\footnote{Remark that in
algebraic instances 2., condition $(d)$ in the ground structure is
crucial, while condition $(b)$ is the crucial one in the former
topological instances 1.}
\\

\noindent 3. Finally, we would like to remark that the joining of
new \textit{geometrical} examples, to the \textit{algebraic} and
\textit{topological} well-known examples above, has been made
possible by a generalization of the assumptions in the ground
structure, done in \cite{X:GenCC}, where a new class of instances is
presented.

\end{document}